\documentclass[11pt,a4paper]{amsart}
\usepackage{amssymb,amsmath,epsfig,graphics,mathrsfs,enumerate,verbatim}
\usepackage[pagebackref,colorlinks=true,linkcolor=blue,citecolor=blue]{hyperref}
\usepackage{fancyhdr}
\usepackage{hyperref}
\hypersetup{
 colorlinks   = true,
 urlcolor     = blue,
 linkcolor    = blue,
 citecolor   = red ,
 bookmarksopen=true
}

\usepackage{amsmath}
\usepackage{amsfonts}
\usepackage{amssymb}
\usepackage{amsthm}
\usepackage{epsfig,graphics,mathrsfs}
\usepackage{graphicx}
\usepackage{dsfont}

\usepackage[usenames, dvipsnames]{color}

\usepackage{hyperref}

\makeatletter
\@namedef{subjclassname@2020}{%
  \textup{2020} Mathematics Subject Classification}
\makeatother

\def \phi {\varphi}

\def \R {\mathbb{R}}

\def \vf{\varphi}

\def \So {\mathscr{S}(\Rm)}

\newcommand{\Rm}{\mathbb R^d}

\newcommand{\p}{\partial}

\newcommand{\la}{\lambda}

\numberwithin{equation}{section}

\newcommand{\beq}{\begin{equation}}
\newcommand{\bea}[1]{\begin{array}{#1} }
\newcommand{\eeq}{ \end{equation}}
\newcommand{\ea}{ \end{array}}

\newcommand{\sa}{\langle}
\newcommand{\da}{\rangle}

\newtheorem{theorem}{Theorem}[section]
\newtheorem{lemma}[theorem]{Lemma}
\newtheorem{proposition}[theorem]{Proposition}
\newtheorem{corollary}[theorem]{Corollary}
\newtheorem{remark}[theorem]{Remark}

\numberwithin{equation}{section}

\begin{document}

\title[]{The Schr\"odinger Ornstein--Uhlenbeck flow:
dispersion, restriction and nonlinear dynamics}

\dedicatory{I dedicate this paper to Carlos Kenig, a deeply influential mathematician.}



\date{}
\begin{abstract}
We study the Schr\"odinger evolution generated by the isotropic
Ornstein--Uhlenbeck operator
\[
\mathscr L=\Delta-\langle x,\nabla\rangle
\]
in $L^2(\mathbb R^d,d\gamma)$, where $d\gamma$ is the invariant
Gaussian measure.  Combining the classical lens correspondence
between the free Schr\"odinger equation and the harmonic oscillator
with Gaussian conjugation, we identify the free Schr\"odinger
structure hidden in the Ornstein--Uhlenbeck evolution.

The main consequence is a dynamic restriction theorem of spacetime
type.  After an explicit transformation of spacetime variables, the
adjoint Ornstein--Uhlenbeck extension operator becomes an ordinary
spacetime Fourier transform restricted to the classical
Schr\"odinger paraboloid.  On the Ornstein--Uhlenbeck side this
corresponds to an explicitly time-dependent Gaussian spacetime
measure.  Thus a restriction geometry which is neither suggested by
the symbol of $\mathscr L$ nor selected by a scaling symmetry of the
Ornstein--Uhlenbeck equation emerges dynamically from the free
Schr\"odinger representation.

The same correspondence identifies the weighted Gaussian spaces in
the homogeneous and inhomogeneous Ornstein--Uhlenbeck Strichartz
estimates and transports the mass-critical nonlinear Schr\"odinger
equation to
\[
i\partial_tu+\mathscr Lu
=
\mu e^{-|x|^2/d}|u|^{4/d}u.
\]
We also obtain a global description of the unitary
Ornstein--Uhlenbeck group, including its oscillatory representation
away from the caustics and its exact values at the caustic times,
together with the associated sharp weighted dispersive estimate and
Hardy-type uncertainty principles.
\end{abstract}

\author{Nicola Garofalo}

\address{School of Mathematical and Statistical Sciences\\ Arizona State University\\ Tempe, AZ 85287-1804}
\vskip 0.2in
\email{Nicola.Garofalo@asu.edu}


\maketitle

\tableofcontents

\section{Introduction}\label{S:intro}

In this work we study the Schr\"odinger evolution generated by the
isotropic Ornstein--Uhlenbeck operator
\[
\mathscr L:=\Delta-\langle x,\nabla\rangle
\]
in
\[
L^2(\R^d,d\gamma),
\qquad
d\gamma(x)=(2\pi)^{-d/2}e^{-|x|^2/2}\,dx.
\]
Our starting point is the Cauchy problem
\begin{equation}\label{cp0}
\begin{cases}
i\partial_tu+\mathscr Lu=0,\\
u(\cdot,0)=\vf,\qquad \vf\in L^2(\R^d,d\gamma).
\end{cases}
\end{equation}
The natural self-adjoint realization of $\mathscr L$ generates a
strongly continuous unitary group
$\{e^{it\mathscr L}\}_{t\in\R}$ on $L^2(\R^d,d\gamma)$.

At the level of its differential expression, \eqref{cp0} looks quite
different from the free Schr\"odinger equation.  The drift destroys
translation invariance, Gaussian measure replaces Lebesgue measure,
and there is no Schr\"odinger scaling.  The central observation of
this paper is that, despite these differences, the Ornstein--Uhlenbeck
evolution contains an exact copy of the dispersive structure of the
free Schr\"odinger equation.  More precisely, two changes of
representation lead to the chain
\medskip
\[
\boxed{
\text{free Schr\"odinger}
\quad\longleftrightarrow\quad
\text{harmonic oscillator}
\quad\longleftrightarrow\quad
\text{Schr\"odinger Ornstein--Uhlenbeck}.
}
\]
\medskip
The first correspondence is the classical lens, or Niederer,
transformation; the second is Gaussian conjugation.  Neither
correspondence separately is new.  Our purpose is instead to follow
some of the principal structures of free Schr\"odinger analysis
through their composition and determine what they become in Gaussian
Ornstein--Uhlenbeck variables.

This point of view produces a coherent picture.  The weighted
Gaussian spaces occurring in Ornstein--Uhlenbeck Strichartz
estimates are forced by the change of representation.  More
unexpectedly, the classical Schr\"odinger paraboloid survives the
transformation and gives a dynamic restriction theorem in Gaussian
spacetime.  At the nonlinear level, the same correspondence selects
a canonical Gaussian mass-critical nonlinearity.  Finally, the
global Ornstein--Uhlenbeck group exhibits a periodic dispersive
geometry with caustics, sharp weighted dispersion, and Hardy-type
uncertainty principles.  These phenomena are different
manifestations of the same underlying free Schr\"odinger structure.

We now describe the correspondence more precisely.  The relation
between the free Schr\"odinger equation and the isotropic harmonic
oscillator goes back to Niederer \cite{Ni1,Ni2}.  In the
normalization used here, if
\[
i\partial_s\psi+\Delta\psi=0,
\]
then, for $|t|<\pi/2$,
\begin{equation}\label{lens-intro}
v(x,t)
:=
(\cos t)^{-d/2}
\exp\left(-i\frac{|x|^2}{4}\tan t\right)
\psi\left(\frac{x}{\cos t},\tan t\right)
\end{equation}
solves
\[
i\partial_tv+
\left(\Delta-\frac{|x|^2}{4}\right)v=0.
\]
Thus $s=\tan t$ compactifies the free time axis into a finite
interval.

The free--harmonic correspondence has a substantial history in
dispersive PDE. We mention here only a few significant sources. Carles used it to relate the mass-critical
nonlinear Schr\"odinger equation with and without an isotropic
harmonic potential \cite{Carles}; Tao subsequently emphasized the lens
transformation as a pseudoconformal compactification of spacetime
and used it to relate scattering, global spacetime bounds, and
well-posedness for the two evolutions \cite{TaoLens}.  Its role in
the variational theory of the mass-critical Strichartz inequality
was exploited by Duyckaerts, Merle and Roudenko \cite{DMR}; see also
\cite{BGTV} for an application of the lens transform to scattering
in $\Sigma$ for mass-subcritical nonlinear Schr\"odinger equations.

The second correspondence is particularly simple.  Set
\[
\omega(x)=e^{-|x|^2/4}.
\]
Then
\[
\omega\mathscr L\omega^{-1}
=
\Delta-\frac{|x|^2}{4}+\frac d2,
\]
and consequently
\[
v(x,t)
=
e^{-idt/2}\omega(x)
\big(e^{it\mathscr L}\vf\big)(x)
\]
solves the harmonic-oscillator equation.  It is the composition of
this Gaussian conjugation with \eqref{lens-intro} that exposes the
free Schr\"odinger structure of the Ornstein--Uhlenbeck flow.

A first manifestation is the Strichartz theory.  Recall that a pair
$(q,r)$, with $2\le q,r\le\infty$, is Schr\"odinger-admissible if
\[
\frac2q+\frac dr=\frac d2,
\]
The endpoint $(q,r,d)=(2,\infty,2)$ is excluded, since the corresponding estimate fails in dimension $2$ \cite{MS98}.
The mixed-norm estimates for the free Schr\"odinger equation are due
to Ginibre and Velo \cite{GV}, while the endpoint theory in dimensions $d\geq3$ was completed by Keel and Tao \cite{KT}.  Transport through \eqref{lens-intro} and
Gaussian conjugation produces the spatial spaces
\[
L^r_\gamma(\omega^{r-2}),
\qquad
\|f\|_{L^r_\gamma(\omega^{r-2})}
=
\left(
\int_{\R^d}|f(x)|^r\omega(x)^{r-2}\,d\gamma(x)
\right)^{1/r},
\]
for $2\le r<\infty$, while at the endpoint we set
\[
\|f\|_{L^\infty_{\gamma,\omega}}:=\|\omega f\|_{L^\infty(dx)}.
\]
With this convention the one-dimensional admissible endpoint $(q,r)=(4,\infty)$ is included; in the displayed estimates below, $L^r_\gamma(\omega^{r-2})$ is understood as $L^\infty_{\gamma,\omega}$ when $r=\infty$.  For finite $r$, the dual spaces are
 $L^{r'}_\gamma(\omega^{r'-2})$.  Thus both the
Gaussian weight and its exponent are determined by the
representation.  The resulting homogeneous and inhomogeneous
estimates have the form
\begin{equation}\label{introStrHom}
\|e^{it\mathscr L}\vf\|_
{L^q_tL^r_\gamma(\omega^{r-2})}
\leq C_{d,q,r}\|\vf\|_{L^2(d\gamma)}
\end{equation}
and
\begin{equation}\label{introStrInhom}
\left\|
\int_0^t e^{i(t-s)\mathscr L}F(s)\,ds
\right\|_
{L^q_tL^r_\gamma(\omega^{r-2})}
\leq
C
\|F\|_
{L^{\widetilde q'}_t
 L^{\widetilde r'}_\gamma
 (\omega^{\widetilde r'-2})},
\end{equation}
for admissible pairs in the usual range.  In this way the
correspondence explains, rather than merely reproduces, the weighted
spaces naturally associated with the Ornstein--Uhlenbeck flow.

If the correspondence did no more than transport norm inequalities,
one might regard it primarily as a convenient change of variables.
The restriction theorem shows that more is encoded in it: the
underlying Fourier geometry is transported as well.  Let
\[
\mathcal P
=
\big\{(\xi,\tau)\in\R^d\times\R:
\tau=-2\pi|\xi|^2\big\}
\]
be the classical Schr\"odinger paraboloid, and consider the free
extension operator
\[
(Eg)(y,s)
=
\int_{\R^d}
e^{2\pi i\langle y,\xi\rangle-4\pi^2is|\xi|^2}
g(\xi)\,d\xi.
\]
Transporting $E$ through the full correspondence defines an
Ornstein--Uhlenbeck extension operator $E_{\rm OU}$.  At the
restriction exponent
\[
p_{\rm res}=\frac{2(d+1)}{d+3},
\qquad
p_{\rm res}'=\frac{2(d+1)}{d-1},
\]
the corresponding spacetime measure is
\begin{equation}\label{OUrestrictionmeasure-intro}
d\nu(x,t)
=
(\cos t)^{\frac{2}{d-1}}
\omega(x)^{\frac{4}{d-1}}\,
d\gamma(x)\,dt,
\qquad |t|<\frac{\pi}{2}.
\end{equation}
Again, the powers in this measure are fixed by the transformation
law.

Taking the adjoint reveals the hidden geometry.  With
\[
s=\tan t,\qquad y=\frac{x}{\cos t},
\]
one obtains
\begin{equation}\label{OUadjoint-intro}
(E_{\rm OU}^*F)(\xi)
=
(2\pi)^{-d/2}
\widehat{\mathcal CF}
\big(\xi,-2\pi|\xi|^2\big),
\end{equation}
where
\begin{equation}\label{OUrestrictiontransform-intro}
\begin{aligned}
(\mathcal CF)(y,s)
={}&
(1+s^2)^{-\frac d4-1-\frac1{d-1}}
e^{-\frac{id}{2}\arctan s}
e^{\frac{is|y|^2}{4(1+s^2)}}\\
&\times
e^{-\frac{d+3}{4(d-1)}
       \frac{|y|^2}{1+s^2}}
F\left(
\frac{y}{\sqrt{1+s^2}},
\arctan s
\right).
\end{aligned}
\end{equation}
Thus an ordinary spacetime Fourier transform, restricted to the
classical paraboloid, is concealed inside the Ornstein--Uhlenbeck
evolution.
For the restriction results we assume $d\ge2$.

\medskip
\noindent
\textbf{Theorem A (Dynamic restriction for the isotropic
Ornstein--Uhlenbeck flow).}
\emph{For every
$F\in C^\infty_0(\R^d\times(-\pi/2,\pi/2))$,
\[
\left\|
(2\pi)^{-d/2}
\widehat{\mathcal CF}
\big(\xi,-2\pi|\xi|^2\big)
\right\|_{L^2(\R^d)}
\leq
C_d
\|F\|_{L^{p_{\rm res}}(d\nu)}.
\]
Equivalently,
\[
\|E_{\rm OU}^*F\|_{L^2(\R^d)}
\leq
C_d\|F\|_{L^{p_{\rm res}}(d\nu)}.
\]
}

\medskip

The point of Theorem A is not the complexity of its proof but the
geometry it exposes.  Neither the differential expression of
$\mathscr L$ nor a scaling symmetry of \eqref{cp0} predicts
$\mathcal P$; in fact, \eqref{cp0} has no scaling symmetry.  The
paraboloid is inherited dynamically from the free equation, while
$d\nu$ records the corresponding deformation of spacetime measure.
In this sense the full correspondence transports not only
Strichartz norms but the Fourier geometry from which the free
Schr\"odinger theory originates.

There are other notions of dynamical restriction for Schr\"odinger
propagators.  In particular, Nicola \cite{NicolaRestriction}
recently established a general principle concerning restriction at
fixed time to curved spatial submanifolds.  The phenomenon in
Theorem A is different: time participates in the spacetime Fourier
transform, and the restriction manifold is the Schr\"odinger
paraboloid in spacetime frequency.

The nonlinear theory provides a second structural test of the
correspondence.  Consider the standard mass-critical equation
\begin{equation}\label{freecriticalNLS-intro}
i\partial_s\psi+\Delta\psi
=
\mu|\psi|^{4/d}\psi.
\end{equation}
Its transport from the free equation to the harmonic oscillator is
classical \cite{Carles,TaoLens}; at the exponent $1+4/d$ the
time-dependent factor generated by the lens transformation cancels.
Gaussian conjugation then gives
\begin{equation}\label{OUcriticalNLS-intro}
i\partial_tu+\mathscr Lu
=
\mu\omega^{4/d}|u|^{4/d}u
=
\mu e^{-|x|^2/d}|u|^{4/d}u.
\end{equation}
Thus the Gaussian factor in \eqref{OUcriticalNLS-intro} is not
introduced in order to make a fixed-point argument work: it is
forced by the exact image of the standard mass-critical NLS in
Ornstein--Uhlenbeck variables.

The same Gaussian power also appears from the sharp diagonal
Strichartz functional.  The lens transform has previously been used
in the variational analysis of the mass-critical Strichartz
inequality \cite{DMR}; after the full correspondence, the
Euler--Lagrange equation for the Gaussian Ornstein--Uhlenbeck
functional contains precisely
\[
\omega^{4/d}|u|^{4/d}u.
\]
The transport of the classical mass-critical equation and the
Gaussian variational structure therefore select the same
nonlinearity.

The corresponding well-posedness theory is obtained by transporting
the classical mass-critical theory.  Scattering at infinite free
time becomes continuation through the endpoints of a Niederer chart,
and Gaussian conjugation transfers this continuation to the
Ornstein--Uhlenbeck variables.  This is the nonlinear counterpart of
the same change of representation that produced
\eqref{introStrHom}--\eqref{introStrInhom} and Theorem A.

It is useful here to compare with the recent work of Dasgupta, Dolai,
Luo and Song \cite{DDLS}.  After the elementary normalization
described in Section \ref{S:strichartz}, their homogeneous and
inhomogeneous linear weighted Strichartz spaces coincide with
\eqref{introStrHom}--\eqref{introStrInhom}.  The nonlinear equations
are different.  At the mass-critical exponent, in the normalization
of the present paper, the nonlinearity considered in \cite{DDLS} is
\[
\mu\omega^{1+4/d}|u|^{4/d}u,
\]
whereas the full free--harmonic--Ornstein--Uhlenbeck correspondence
gives
\[
\mu\omega^{4/d}|u|^{4/d}u.
\]
The latter is the exact image of the standard mass-critical NLS.

The correspondence also fits naturally with the intrinsic global
analysis of the Ornstein--Uhlenbeck group.  Away from the caustic
times $t=k\pi$ the propagator admits an explicit oscillatory
representation, while at the caustics the unitary evolution itself
remains regular.  In particular,
\[
e^{i\pi\mathscr L}=\mathcal R,
\qquad
e^{i2\pi\mathscr L}=I,
\qquad
(\mathcal R\vf)(x)=\vf(-x).
\]
Thus the imaginary Ornstein--Uhlenbeck evolution is periodic, with
reflection at half-period.  At the quarter periods it becomes,
after Gaussian conjugation, a Fourier transform.  The fundamental
solution on the first period, the sharp weighted dispersive estimate,
and the associated Hardy-type uncertainty principles first appeared
in our earlier preprint \cite{GarOUpre}; the global formulation used
here identifies the evolution at the caustics and makes the periodic
structure explicit.  Related dynamical Hardy uncertainty principles
for Schr\"odinger equations with real drift have recently been
obtained by Garofalo and Lunardi \cite{GarLun} in the broader
H\"ormander setting.

Taken together, these results show that the free Schr\"odinger
correspondence is more than a device for transferring individual
estimates.  In Gaussian Ornstein--Uhlenbeck variables it determines
the natural Strichartz weights, reveals a hidden paraboloid,
identifies the mass-critical nonlinear equation, and organizes the
global dispersive and uncertainty structure of the flow.

The paper is organized as follows.  Section \ref{S:senzadrift}
develops the intrinsic linear Schr\"odinger Ornstein--Uhlenbeck flow,
including its global representation, caustics, and sharp dispersive
estimate.  Section \ref{S:app} establishes the Gaussian
correspondence with the harmonic oscillator and records the Niederer
transformation in the normalization used here.  Section
\ref{S:strichartz} derives the weighted homogeneous and
inhomogeneous Strichartz estimates.  The dynamic restriction theorem
and the emergence of the hidden paraboloid are developed in Section
\ref{S:restriction}.  Section \ref{S:nonlinear} discusses the
critical variational structure and the mass-critical
Ornstein--Uhlenbeck dynamics, including the well-posedness theory
obtained by transport.  Section \ref{S:uncertainty} contains the
Hardy-type uncertainty principles.

\medskip

\noindent\textbf{Acknowledgment.}
I am grateful to Federico Buseghin for carefully reading an earlier
version of the manuscript, for pointing out an oversight in the phase
of the lens transformation, and for several valuable comments and
references which helped improve the presentation of the paper.

\section{The Schr\"odinger Ornstein--Uhlenbeck flow}\label{S:senzadrift}

We begin by recording the functional setting and the precise statements that will be proved in this section.  The operator $\mathscr L$ is self-adjoint in $L^2(\Rm,d\gamma)$, with its natural self-adjoint realization.  Consequently, by Stone's theorem, see \cite[Theorem 1, p.~345]{Yo}, it generates a strongly continuous one-parameter group of unitary operators
\begin{equation}\label{OUunitarygroup}
\{e^{it\mathscr L}\}_{t\in\mathbb R}
\quad\text{in }L^2(\Rm,d\gamma),
\end{equation}
and
\begin{equation}\label{OUunitarity}
\|e^{it\mathscr L}\vf\|_{L^2(d\gamma)}
=\|\vf\|_{L^2(d\gamma)},\qquad t\in\mathbb R.
\end{equation}
For the explicit analysis of the group, given $\vf$ we set
\begin{equation}\label{psi}
\psi(x)=e^{-|x|^2/4}\vf(x).
\end{equation}
Then
\[
\psi\in L^2(\Rm)\quad\Longleftrightarrow\quad \vf\in L^2(\Rm,d\gamma),
\]
and
\begin{equation}\label{phipsi}
\|\psi\|_{L^2(\Rm)}=(2\pi)^{d/4}\|\vf\|_{L^2(\Rm,d\gamma)}.
\end{equation}
We use the dense class
\begin{equation}\label{K}
\mathscr K(\Rm)=\{\vf\in C^\infty(\Rm):\psi\in\mathscr S(\Rm)\}.
\end{equation}
The explicit formulas are first established on $\mathscr K(\Rm)$; density and the unitarity of \eqref{OUunitarygroup} then determine the $L^2(d\gamma)$ evolution uniquely.

Our global representation is the following.  In the next statement $\lfloor s\rfloor$ denotes the greatest integer not exceeding $s$.

\begin{proposition}\label{P:OUim}
Let $\vf\in\mathscr K(\Rm)$ and, for $t\in\R\setminus\pi\mathbb Z$, set
\[
k(t)=\left\lfloor\frac{t}{\pi}\right\rfloor.
\]
Then, for every $x\in\Rm$ and $t\in\R\setminus\pi\mathbb Z$,
\begin{align}\label{OUim}
\left(e^{it\mathscr L}\vf\right)(x)
={}&\frac{(4\pi)^{-\frac d2}e^{\frac{idt}{2}}
 e^{-\frac{i\pi d}{4}(2k(t)+1)}}{|\sin t|^{\frac d2}} \notag\\
&\times\int_{\Rm}
\exp\!\left\{
\frac{i\,\operatorname{sgn}(\sin t)}{4|\sin t|}
\left(e^{it}|y|^2+e^{-it}|x|^2-2\langle x,y\rangle\right)
\right\}\vf(y)\,dy.
\end{align}
At the caustic times,
\begin{equation}\label{OUim-caustic}
\left(e^{ik\pi\mathscr L}\vf\right)(x)=\vf((-1)^k x),\qquad k\in\mathbb Z.
\end{equation}
These formulas give the solution of \eqref{cp0} in $\Rm\times\R$.
\end{proposition}

The singularities of \eqref{OUim} are caustics of the oscillatory representation, not singularities of the unitary evolution.  At the quarter periods the group is instead a Gaussian-conjugated Fourier transform.  Let $(\mathcal R\vf)(x)=\vf(-x)$.

\begin{corollary}[Special times]\label{C:quarter}
For every $\vf\in\mathscr K(\Rm)$ and every $k\in\mathbb Z$,
\begin{equation}\label{causticOU}
e^{ik\pi\mathscr L}\vf(x)=\vf((-1)^kx)=\mathcal R^k\vf(x),
\end{equation}
and hence
\begin{equation}\label{OUperiodicity}
e^{i\pi\mathscr L}=\mathcal R,\qquad e^{i2\pi\mathscr L}=I.
\end{equation}
Moreover,
\begin{equation}\label{quarterOU}
e^{-|x|^2/4}\left(e^{i(\frac\pi2+k\pi)\mathscr L}\vf\right)(x)
=(4\pi)^{-d/2}\,
\mathscr F\!\left(e^{-|\cdot|^2/4}\vf\right)
\left(\frac{(-1)^kx}{4\pi}\right).
\end{equation}
By density, \eqref{causticOU}--\eqref{quarterOU} extend as identities in $L^2(d\gamma)$.
\end{corollary}

The points $t=k\pi$ are genuine caustics of the oscillatory kernel, whereas $t=\pi/2+k\pi$ are the boundaries of the individual Niederer charts and correspond to infinite free Schr\"odinger time.  This distinction reappears in the global nonlinear theory.

\begin{remark}\label{R:stable}
The global representation in Proposition \ref{P:OUim}, together with \eqref{causticOU}, shows that
\begin{equation}\label{nice}
e^{it\mathscr L}:\mathscr K(\Rm)\longrightarrow\mathscr K(\Rm),
\qquad t\in\mathbb R.
\end{equation}
Thus $\mathscr K(\Rm)$ is invariant under the full Schr\"odinger Ornstein--Uhlenbeck group.
\end{remark}

We use the Fourier-transform normalization \eqref{ft}, for which Plancherel's theorem is unitary, and set
\[
J^+=(0,\pi),\qquad J^-=(\pi,2\pi),\qquad J=J^+\cup J^-.
\]

\begin{proposition}\label{P:semigroup}
Let $\vf\in\mathscr K(\Rm)$. Then for every $t\in J^+$,
\begin{equation}\label{final}
e^{-|x|^2/4}\left(e^{it\mathscr L}\vf\right)(x)
=(4\pi)^{-d/2}\frac{e^{idt/2}}{e^{i\pi d/4}(\sin t)^{d/2}}
 e^{i\cot t\,|x|^2/4}
\mathscr F\!\left(e^{i\cot t\,|\cdot|^2/4}\psi\right)
\left(\frac{x}{4\pi\sin t}\right),
\end{equation}
where $\psi$ is defined in \eqref{psi}.
\end{proposition}

As a consequence we obtain the sharp weighted dispersive estimate.
\begin{proposition}\label{P:disp}
Let $\vf\in\mathscr K(\Rm)$ and $t\notin\pi\mathbb Z$. For every $1\le p\le2$,
\begin{equation}\label{disp}
\|e^{-|\cdot|^2/4}e^{it\mathscr L}\vf\|_{L^{p'}(\Rm)}
\le
\left(\frac{p^{1/p}}{{p'}^{1/p'}}\right)^{d/2}
\frac{\|e^{-|\cdot|^2/4}\vf\|_{L^p(\Rm)}}
{(4\pi|\sin t|)^{d(\frac12-\frac1{p'})}},
\end{equation}
where $1/p+1/p'=1$.  The estimate is optimal in the sense that it cannot hold for $2<p\le\infty$, and equality is attained by $\vf(x)=e^{-\alpha|x|^2+|x|^2/4}$, $\Re\alpha>0$.
\end{proposition}

The explicit representation and sharp weighted dispersive estimate above, together with the uncertainty principles below, appeared in an earlier version of this work \cite{GarOUpre}.  The uncertainty principles belong to the broad circle of Hardy and dynamical Hardy principles for Schr\"odinger evolutions; see Hardy's original theorem \cite{Ha}, the works \cite{EKPVcpde,EKPVjems,EKPVduke,CEKPV,EKPVjlms}, and \cite{CP,SST,Fo,Veluma,FM}.  For Gaussian decay and related uncertainty questions for the harmonic oscillator, see also \cite{CF,KJO,RR}.  We record their statements here for completeness.

\begin{proposition}\label{P:main}
Assume that for some $a,b>0$,
\begin{equation}\label{L2}
\|e^{a|\cdot|^2}\vf\|_{L^2(\Rm,d\gamma)}
+\|e^{b|\cdot|^2}e^{is\mathscr L}\vf\|_{L^2(\Rm,d\gamma)}<\infty.
\end{equation}
If $ab\sin^2s\ge1/16$, then $\vf\equiv0$.
\end{proposition}

The hypothesis automatically excludes $s\in\pi\mathbb Z$.  Combining Proposition \ref{P:semigroup} with Hardy's uncertainty principle also gives the pointwise counterpart.

\begin{proposition}\label{P:main2}
Suppose that for some $C,a,b>0$ and every $x\in\Rm$,
\begin{equation}\label{Linfty}
|e^{-|x|^2/4}\vf(x)|\le Ce^{-a|x|^2},\qquad
|e^{-|x|^2/4}(e^{is\mathscr L}\vf)(x)|\le Ce^{-b|x|^2}.
\end{equation}
If $ab\sin^2s\ge1/16$, then $\vf\equiv0$.
\end{proposition}

\begin{proof}[Proof of Proposition \ref{P:OUim}]

We begin with a simple, but critical observation. Suppose that $v$ and $f$ are connected by the relation
\begin{equation}\label{drift0}
v(x,t) = f(e^{i t} x,t).
\end{equation}
Then, $f$ is a solution of the Cauchy problem \eqref{cp0} if and only if $v$ solves the problem
\begin{equation}\label{cpGsenzadrift00}
\begin{cases}
\p_t v - i U'(t) \Delta v = 0,
\\
v(x,0) = \vf(x),
\end{cases}
\end{equation} 
where we have let
\begin{equation}\label{Qt0}
U(t) =  \int_0^t e^{-2is} ds = \frac{1-e^{-2it}}{2i} = e^{-it} \frac{e^{it}-e^{-it}}{2i} = e^{-it} \sin t.
\end{equation}
To prove that $v$ solves \eqref{cpGsenzadrift00}, we argue as follows. The chain rule gives from \eqref{drift0}
\[
v_t(x,t) = i e^{it} \sa x,\nabla f(e^{it} x,t)\da + f_t(e^{it} x,t).
\]
On the other hand, the PDE in \eqref{cp0} gives
\[
f_t(e^{it} x,t) = i \Delta f(e^{it} x,t) - i e^{it} \sa x,\nabla f(e^{it} x,t)\da.
\]
Combining the latter two equations, we infer that $v$ solves 
\[
v_t(x,t) = i \Delta f(e^{it} x,t).
\]
Next, differentiating \eqref{drift0} we find
\[
\Delta v(x,t) = e^{2it} \Delta f(e^{it} x,t).
\]
We thus conclude that 
\[
v_t(x,t) = i e^{-2it} \Delta v(x,t) = i U'(t) \Delta v(x,t),
\]
where in the second equality we have used \eqref{Qt0}.
Summarising, the function $v$ solves the problem \eqref{cpGsenzadrift00}.
To find a representation formula for the latter, we use the Fourier transform.
Supposing that $v$ be a solution, we define
\begin{equation}\label{pFTv}
\hat v(\xi,t) = \mathscr F(v)(\xi,t) = \int_{\Rm} e^{-2\pi i\sa \xi,x\da} v(x,t) dx.
\end{equation}
Then \eqref{cpGsenzadrift00} is transformed into
\begin{equation}\label{cpGsenzadrift0}
\begin{cases}
\p_t \hat v + 4\pi^2 i U'(t)|\xi|^2 \hat v = 0,
\\
\hat v(\xi,0) = \hat \vf(\xi),
\end{cases}
\end{equation} 
whose solution is given by
\begin{equation}\label{hatv}
\hat v(\xi,t) = \hat \vf(\xi) e^{-4\pi^2 i U(t) |\xi|^2}.
\end{equation} 
Note that, with $U(t)$ as in \eqref{Qt0}, for every $t\in J$ the matrix $Q(t) = U(t) I_d$ is invertible. Moreover, we have
\begin{equation}\label{iQ}
i Q(t) = i U(t) I_d = i(\cos t - i \sin t) \sin t\ I_d = \sin^2 t\ I_d + i \frac{\sin{2t}}2 I_d. 
\end{equation}
We now invoke \cite[Theorem 7.6.1]{Hobook}, which we formulate as follows: Let $A\in G\ell(\mathbb C,d)$ be such that $A^\star = A$ and $\Re A \ge 0$. Then 
\begin{equation}\label{gengaussi2}
\mathscr F\left(\frac{(4\pi)^{-\frac{d}{2}}}{\sqrt{\operatorname{det} A}} e^{- \frac{\sa A^{-1}\cdot,\cdot\da}{4}}\right)(\xi) =
e^{- 4 \pi^2  \sa A\xi,\xi\da},
\end{equation}
where $\sqrt{\operatorname{det} A}$ is the unique analytic branch such that $\sqrt{\operatorname{det} A}>0$ when $A$ is real. 
If in \eqref{gengaussi2} we take 
\[
A = i Q(t) = i U(t) I_d =  i e^{-it} \sin t\ I_d,
\]
with  $t\in J^+$, then \eqref{iQ} gives $\Re A = \sin^2 t\ I_d \ge 0$, and
\[
A^{-1} = - i  \frac{e^{it}}{\sin t} I_d.
\]
We thus find 
\begin{equation}\label{larsetto2}
e^{-4\pi^2 i U(t) |\xi|^2}  = \mathscr F\left(\frac{e^{\frac{idt}2}}{e^{\frac{i\pi d}{4}} (\sin t)^{\frac d2}}(4\pi
)^{-\frac{d}{2}} e^{i e^{it}\frac{|\cdot|^2}{4 \sin t}}\right)(\xi).
\end{equation}
From \eqref{hatv} and \eqref{larsetto2} we conclude that for every $x\in \Rm$ and $t\in J^+$
\begin{equation}\label{hatv2}
v(x,t) = \frac{e^{\frac{idt}2}}{e^{\frac{i\pi d}{4}} (\sin t)^{\frac d2}}(4\pi
)^{-\frac{d}{2}} \int_{\Rm} e^{i e^{it}\frac{|y-x|^2}{4 \sin t}} \vf(y) dy.
\end{equation} 
Finally, keeping \eqref{drift0} in mind, after some elementary algebraic manipulations, we obtain \eqref{OUim} for $0<t<\pi$, where $k(t)=0$. The corresponding formula on $(\pi,2\pi)$ follows by the same argument, observing that in this interval $A=e^{i\frac{3\pi}{2}}e^{-it}|\sin t|I_d$.

We now extend the representation to every noncaustic time. By the Hermite spectral resolution of $\mathscr L$ (recalled in Subsection \ref{S:caustics}) one has
\begin{equation}\label{parity-pre}
e^{i\pi\mathscr L}=\mathcal R,\qquad (\mathcal R g)(x)=g(-x),
\end{equation}
and therefore $e^{ik\pi\mathscr L}=\mathcal R^k$ for every $k\in\mathbb Z$. Given $t\notin\pi\mathbb Z$, write
\[
t=k\pi+\tau,\qquad k=\left\lfloor\frac{t}{\pi}\right\rfloor,
\qquad 0<\tau<\pi.
\]
The group property and \eqref{parity-pre} give
\[
e^{it\mathscr L}\vf(x)
=\big(e^{i\tau\mathscr L}\vf\big)((-1)^k x).
\]
Applying the already established formula on $(0,\pi)$ with $(-1)^kx$ in place of $x$, and using
\[
\sin\tau=|\sin t|,\qquad (-1)^k=\operatorname{sgn}(\sin t),
\qquad e^{i\tau}=(-1)^k e^{it},
\]
together with
\[
e^{\frac{id\tau}{2}}e^{-\frac{i\pi d}{4}}
=e^{\frac{idt}{2}}e^{-\frac{i\pi d}{4}(2k+1)},
\]
yields exactly \eqref{OUim}. At $t=k\pi$, \eqref{parity-pre} gives \eqref{OUim-caustic}. This also shows directly that the apparent singularities of \eqref{OUim} at $\pi\mathbb Z$ belong to the kernel representation, not to the evolution operator.

\end{proof}

\subsection{Caustics and the global unitary group}\label{S:caustics}

The factors $(\sin t)^{-d/2}$ and $\cot t$ in \eqref{OUim} make the times $t=k\pi$ look singular.  They are singular only for this nondegenerate oscillatory-Gaussian representation.  The operator $\mathscr L=\Delta-x\cdot\nabla$ is self-adjoint in $L^2(d\gamma)$ and its Hermite-polynomial eigenfunctions satisfy
\[
\mathscr L H_\alpha=-|\alpha|H_\alpha,\qquad \alpha\in\mathbb N^d.
\]
Since the Hermite polynomials form a complete orthogonal system in $L^2(d\gamma)$, spectral calculus gives
\[
e^{it\mathscr L}H_\alpha=e^{-it|\alpha|}H_\alpha.
\]
At $t=\pi$, using $H_\alpha(-x)=(-1)^{|\alpha|}H_\alpha(x)$, we obtain
\begin{equation}\label{parity}
e^{i\pi\mathscr L}f(x)=f(-x),\qquad f\in L^2(d\gamma),
\end{equation}
and consequently
\begin{equation}\label{periodicity}
e^{i2\pi\mathscr L}=I.
\end{equation}
Thus the unitary evolution is defined and regular for every real time.  At $t=k\pi$ the kernel in \eqref{OUim} concentrates, in the distributional sense, onto the graph $y=(-1)^kx$.  In the terminology of oscillatory-integral and Fourier-integral-operator theory these are \emph{caustic times}: the projection underlying the nondegenerate quadratic phase degenerates, while the evolution operator itself remains well defined.  This distinction is also the reason that a free-type global dispersive estimate cannot hold: the flow is recurrent rather than globally dispersive.

It may be of interest to compare \eqref{OUim} with the well-known Mehler representation (see, for instance, \cite{OU,Bo2,Bo,CMG,StroockPDE,LMP}) 
\begin{align}\label{gustavo}
u(x,t) & = (4\pi)^{- \frac d2} e^{d t \sqrt \omega} \left(\frac{2\sqrt \omega}{\sinh(2t\sqrt \omega)}\right)^{\frac d2}
\\
& \times \int_{\Rm} \exp\left( -  \frac{\sqrt \omega}{2 \sinh(2t\sqrt \omega)} |e^{t\sqrt \omega} y - e^{-t\sqrt \omega} x|^2\right) \vf(y) dy
\notag
\end{align}
for the solution of the Cauchy problem for the Ornstein-Uhlenbeck operator
\begin{equation}\label{cpou}
\begin{cases}
u_t - \Delta u + 2 \sqrt \omega \langle x,\nabla u\rangle  = 0,\ \ \ \ \ \omega>0,
\\
u(x,0) = \vf(x).
\end{cases}
\end{equation}
 If one takes $\omega = \frac 14$, keeping in mind that $\sinh it = i \sin t$, then it is clear that by \emph{formally} substituting $t\to it$ in \eqref{gustavo}, one obtains the case $t\in J^+$ of \eqref{OUim}. Such formal manipulation is reminiscent of the physicist' Wick rotation, see \cite[Section 3]{Wi}.

\begin{proof}[Proof of Proposition \ref{P:semigroup}]
To further unravel \eqref{OUim}, and also to better clarify the role of the class $\mathscr K(\Rm)$ in \eqref{K}, note that if for $t\in J^+$ we expand
\begin{equation}\label{expa}
\frac{|e^{it/2}y-e^{-it/2} x|^2}{4 \sin t} = \frac{e^{it}|y|^2 + e^{-it}|x|^2 - 2\sa x,y\da}{4 \sin t},
\end{equation}
we find 
\begin{equation}\label{OUim4}
\left(e^{it\mathscr L}\vf\right)(x) = \frac{(4\pi
)^{-\frac{d}{2}}e^{\frac{idt}2}}{e^{\frac{i\pi d}{4}} (\sin t)^{\frac d2}} \int_{\Rm} e^{i \frac{e^{it}|y|^2 + e^{-it}|x|^2 - 2\sa x,y\da}{4 \sin t}} \vf(y) dy.
\end{equation}
The change of variable $y = 4 \pi \sin t\ z$ in the integral in \eqref{OUim4} gives
\begin{align*}
& \int_{\Rm} e^{- i \frac{\sa x,y\da}{2 \sin t}} e^{i \frac{e^{it}|y|^2 + e^{-it}|x|^2}{4 \sin t}} \vf(y) dy = (4 \pi \sin t)^d e^{i \frac{e^{-it}|x|^2}{4 \sin t}} \int_{\Rm} e^{- 2 \pi i \sa x,z\da} e^{i \frac{e^{it}  |4 \pi \sin t\ z|^2}{4 \sin t}} \vf(4 \pi \sin t\ z) dz
\\
& = (4 \pi \sin t)^d e^{\frac{|x|^2}{4}} e^{i \frac{\cot t |x|^2}{4}} \int_{\Rm} e^{- 2 \pi i \sa x,z\da} e^{i \frac{\cot t |4 \pi \sin t\ z|^2}{4}} e^{- \frac{|4 \pi \sin t\ z|^2}{4}} \vf(4 \pi \sin t\ z) dz.
\end{align*}
Keeping \eqref{psi} in mind, 
we thus obtain from the above integral  
\begin{align*}
& \int_{\Rm} e^{- i \frac{\sa x,y\da}{2 \sin t}} e^{i \frac{e^{it}|y|^2 + e^{-it}|x|^2}{4 \sin t}} \vf(y) dy = (4 \pi \sin t)^d e^{\frac{|x|^2}{4}} e^{i \frac{\cot t |x|^2}{4}} \mathscr F\left(\delta_{4 \pi \sin t}\ e^{i \frac{\cot t |\cdot|^2}4} \psi\right)(x)
\\
& = e^{\frac{|x|^2}{4}} e^{i \frac{\cot t |x|^2}{4}} \mathscr F\left(e^{i \frac{\cot t |\cdot|^2}4} \psi\right)(\frac{x}{4 \pi \sin t}),
\end{align*}
where we have denoted by $\delta_\la f(x) = f(\la x)$ the action of the dilation operator on a function $f$.
Going back to \eqref{OUim4}, we have finally established \eqref{final}.

\end{proof}

By Proposition \ref{P:OUim} and Remark \ref{R:stable}, the full group
$e^{it\mathscr L}$ preserves $\mathscr K(\Rm)$ and is unitary in
$L^2(\Rm,d\gamma)$. Formula \eqref{final} further unveils the
intertwining between the group $e^{it\mathscr L}$ and the Fourier
transform. 
We now use Proposition \ref{P:semigroup} to provide the 

\begin{proof}[Proof of Proposition \ref{P:disp}]
We first rewrite \eqref{final} in the following fashion
\begin{equation}\label{sfinalicchio}
\mathscr F\left(e^{i \frac{\cot t |\cdot|^2}4} \psi\right)(\frac{x}{4 \pi \sin t}) = (4\pi
)^{\frac{d}{2}}  \frac{e^{\frac{i\pi d}{4}}}{e^{\frac{idt}2}} (\sin t)^{\frac d2} e^{-i \frac{\cot t |x|^2}{4}} e^{- \frac{|x|^2}{4}} \left(e^{it\mathscr L}\vf\right)(x).
\end{equation}
The identity \eqref{sfinalicchio} has the following direct consequence 
\begin{equation}\label{sfinalicchietto}
\left|\mathscr F\left(e^{i \frac{\cot t |\cdot|^2}4} \psi\right)(\frac{x}{4 \pi \sin t})\right| = (4\pi
)^{\frac{d}{2}} |\sin t|^{\frac d2} e^{- \frac{|x|^2}{4}} |(e^{it\mathscr L}\vf)(x)|.
\end{equation}
If now $1\le p \le 2$, recall that in his celebrated paper \cite{Be} Beckner computed the sharp constant in the Hausdorff-Young inequality, and proved that for any $\psi\in L^p(\Rm)$ one has
\begin{equation}\label{HY}
\left(\int_{\Rm}|\mathscr F \psi(y)|^{p'} dy\right)^{\frac{1}{p'}} \le \left(\frac{p^{1/p}}{{p'}^{1/p'}}\right)^{\frac d2} \left(\int_{\Rm}|\psi(y)|^p dy\right)^{\frac 1p}.
\end{equation}
He also showed that equality is attained in \eqref{HY} if and only if $\psi$ is a Gaussian. The change of variable $y = \frac{x}{4 \pi \sin t}$ in \eqref{HY} gives
\begin{equation*}
\left(\int_{\R^d}|\mathscr F \psi(\frac{x}{4 \pi \sin t})|^{p'} dx\right)^{\frac{1}{p'}} \le \left(\frac{p^{1/p}}{{p'}^{1/p'}}\right)^{\frac d2} (4\pi |\sin t|)^{\frac{d}{p'}}\left(\int_{\Rm}|\psi(y)|^p dy\right)^{\frac 1p}.
\end{equation*}
We then turn to the proof of \eqref{disp}. Let $\vf\in \mathscr K(\Rm)$ and $t\in J^+$. We have from \eqref{sfinalicchietto} 
\begin{align*}
& \left(\int_{\R^d}|e^{- \frac{|x|^2}{4}} (e^{it\mathscr L}\vf)(x)|^{p'} dx\right)^{\frac{1}{p'}} = (4\pi
)^{-\frac{d}{2}} |\sin t|^{-\frac d2} \left(\int_{\Rm} \left|\mathscr F\left(e^{i \frac{\cot t |\cdot|^2}4} \psi\right)(\frac{x}{4 \pi \sin t})\right|^{p'} dx\right)^{\frac{1}{p'}} 
\\
& \le (4\pi
)^{-\frac{d}{2}} |\sin t|^{-\frac d2} \left(\frac{p^{1/p}}{{p'}^{1/p'}}\right)^{\frac d2} (4\pi |\sin t|)^{\frac{d}{p'}}\left(\int_{\Rm} |\psi(x)|^{p} dx\right)^{\frac{1}{p}}. 
\end{align*}
Keeping in mind that $\psi(x) = e^{- \frac{|x|^2}{4}} \vf(x)$, see \eqref{psi}, we reach the desired conclusion \eqref{disp} from the latter inequality.

\end{proof}

We will  return to more general dispersive estimates for the group $e^{it\mathscr L}$ in a future study. 
 


\section{The harmonic oscillator and the Niederer transform}\label{S:app}

We begin by making explicit the link between the Ornstein--Uhlenbeck operator and the quantum-mechanical harmonic oscillator.  This conjugation is also the natural point at which the classical lens transform enters the discussion. The main tool is the nonlinear Schr\"odinger equation of Riccati type \eqref{riccati} in Lemma \ref{L:ouho}. Consider the harmonic oscillator \footnote{This operator is usually defined as $H = \Delta - |x|^2$. We are using the $1/4$ normalisation in order not to have to change in $\Delta - 2 \sa v,\nabla\da$ that of the Ornstein-Uhlenbeck operator in the statement of Proposition \ref{P:connect} below.}  in $\Rm$
\begin{equation}\label{H}
H = \Delta - \frac{|x|^2}4
\end{equation}
and the Cauchy problem in $\Rm\times (0,\infty)$ for the associated Schr\"odinger operator
\begin{equation}\label{cpHO}
\p_t u - i H u = 0,\ \ \ \ \ \ \ \ u(x,0) = u_0(x),
\end{equation}
where the initial datum $u_0$ will be taken e.g. in $\mathscr S(\Rm)$. 
The following lemma establishes a general principle, one interesting consequence of which is that it allows to connect \eqref{cpHO} to the problem \eqref{cp0}, and in fact show that they are equivalent. The functions $\Phi$ and $h$ in its statement are assumed complex-valued.

\begin{lemma}\label{L:ouho}
Let $\Phi\in C(\R^{d+1})$ and $h\in C^2(\R^{d+1})$ be connected by the following  nonlinear Schr\"odinger equation
\begin{equation}\label{riccati}
i h_t + \Delta h - |\nabla h|^2  = \Phi.
\end{equation}
Then $u$ solves the partial differential equation
\begin{equation}\label{pdeho}
P u = i(\Delta u + \Phi u) - u_t = 0
\end{equation}
if and only if $f$ defined by the transformation
\begin{equation}\label{genexp}
u(x,t) = e^{-h(x,t)} f(x,t),
\end{equation} 
solves the equation
\begin{equation}\label{PDEou}
i(\Delta f - 2 \langle\nabla h,\nabla f\rangle) - f_t = 0.
\end{equation}
\end{lemma}

\begin{proof}
With $u$ as in \eqref{genexp}, we find
\begin{align*}
P u & = i \Delta(e^{-h} f) + i \Phi e^{-h} f - (e^{-h} f)_t
\\
& = i f \Delta(e^{-h}) + i e^{-h} \Delta f + 2 i \langle\nabla(e^{-h}),\nabla f\rangle + i \Phi e^{-h} f
 - (e^{-h})_t f -  e^{-h} f_t
\\
& = i e^{-h} f |\nabla h|^2 - i e^{-h} f \Delta h + i e^{-h} \Delta f - 2 i e^{-h} \langle\nabla h,\nabla f\rangle + i \Phi e^{-h} f
 - (e^{-h})_t f -  e^{-h} f_t
 \\
 & = e^{-h}\left\{i \left[\Phi - i h_t - (\Delta h - |\nabla h|^2 )\right] f + i \Delta f - 2 i \langle\nabla h,\nabla f\rangle - f_t\right\}.
\end{align*}
This computation proves that if $h$ and $\Phi$ solve \eqref{riccati}, then $u$ solves \eqref{pdeho} if and only if $f$ is a solution of \eqref{PDEou}.

\end{proof}

With Lemma \ref{L:ouho} in hand, we now return to \eqref{H} and prove the following result.
\begin{proposition}\label{P:connect}
A function $u$ solves the Cauchy problem \eqref{cpHO} if and only if the function
\begin{equation}\label{fu}
f(x,t) =  u(x,t)\ e^{\frac{|x|^2}4 + i \frac d2 t}
\end{equation}
solves \eqref{cp0} with $f(x,0) = \vf(x) = u_0(x)\ e^{\frac{|x|^2}4}$.
\end{proposition}

\begin{proof}
It is clear from \eqref{pdeho} that, in order to obtain from it the PDE in \eqref{cpHO}, we need  $\Phi(x,t) = - \frac{|x|^2}4$. With this choice, we look for a function $h(x,t)$ that is connected to such $\Phi$ by the equation \eqref{riccati}. A natural ansatz is $h(x,t) = A |x|^2 + B t$, with $A, B\in \mathbb C$ to be determined. Since $\Delta h = 2d A$, $h_t = B$ and $|\nabla h|^2 = 4 A^2 |x|^2$, to satisfy \eqref{riccati} we want
\[
i B + 2d A - 4 A^2 |x|^2 = - \frac{|x|^2}4,
\]
which holds iff $A = \frac{1}4, B =  i \frac d2$, and thus
\begin{equation}\label{h}
h(x,t) = \frac{|x|^2}4 + i \frac d2 t.
\end{equation}
With such choice of $h(x,t)$, the equation \eqref{genexp} in Lemma \ref{L:ouho} shows that $f(x,t)$ defined in \eqref{fu} solves the Cauchy problem \eqref{cp0}, with $f(x,0) = \vf(x) = u_0(x)\ e^{\frac{|x|^2}4}$. The ``if and only if" character of the statement is obvious.

\end{proof}

If $u_0\in \So$, then it is clear that $e^{\frac{|\cdot|^2}4} u_0\in \mathscr K(\Rm)$. According to Proposition \ref{P:connect}, we can express the group $e^{itH}$ by the formula
\begin{equation}\label{eitH}
e^{itH} u_0(x) = e^{-\frac{|x|^2}4 - i \frac d2 t} e^{it\mathscr L}(e^{\frac{|\cdot|^2}4} u_0)(x).
\end{equation}
Applying \eqref{OUim}, we infer from \eqref{eitH}.

\begin{corollary}\label{C:HOim}
Given $u_0\in \So$, for $t\in J^+$ one has
\begin{equation}\label{eitHg}
e^{itH} u_0(x) =  \frac{(4\pi
)^{-\frac{d}{2}} e^{-\frac{|x|^2}4}}{e^{\frac{i\pi d}{4}} (\sin t)^{\frac d2}} \int_{\Rm} e^{i \frac{|e^{it/2} y-e^{-it/2} x|^2}{4 \sin t}} e^{\frac{|y|^2}4} u_0(y) dy.
\end{equation}
\end{corollary}
Although with a different expression, formula \eqref{eitHg} is well-known to workers in harmonic analysis, see for instance \cite[Eq. (4.13), p.85]{Vel1}, or \cite[Eq. (6), p.164]{ST}.
Using \eqref{expa} in \eqref{eitHg}, we thus obtain the following counterpart of Proposition \ref{P:semigroup}.

\begin{corollary}\label{C:semigroup}
Given $u_0\in \mathscr S(\Rm)$, let $u(x,t) = e^{itH}u_0(x)$. Then for every $t\in J^+$ one has
\begin{equation}\label{finalino}
u(x,t) =  \frac{(4\pi
)^{-\frac{d}{2}}}{e^{\frac{i\pi d}{4}} (\sin t)^{\frac d2}} e^{i \frac{\cot t |x|^2}{4}} \mathscr F\left(e^{i \frac{\cot t |\cdot|^2}4} u_0\right)(\frac{x}{4 \pi \sin t}).
\end{equation}
\end{corollary}

\section{Weighted Strichartz estimates}\label{S:strichartz}

The local equivalence between the free Schr\"odinger equation and the isotropic harmonic oscillator goes back to Niederer \cite{Ni1,Ni2}.  In the modern dispersive PDE literature the corresponding change of variables is usually called the \emph{lens transform}.  Since this transformation is also relevant to the interpretation of the weighted Strichartz estimates subsequently obtained for the isotropic Ornstein--Uhlenbeck equation, we give the complete calculation in the normalization used throughout this paper.

We begin with the free Schr\"odinger equation
\begin{equation}\label{freeLens}
 i\p_s\psi+\Delta_y\psi=F(y,s),\qquad (y,s)\in\Rm\times\mathbb R.
\end{equation}
For $|t|<\pi/2$ put
\begin{equation}\label{Niederer}
 s=\tan t,\qquad y=\frac{x}{\cos t},
\end{equation}
and define
\begin{equation}\label{lensOurNorm}
 v(x,t)=(\cos t)^{-d/2}\exp\left(-i\frac{|x|^2}{4}\tan t\right)
 \psi\left(\frac{x}{\cos t},\tan t\right).
\end{equation}
The choice of the factor $(\cos t)^{-d/2}$ and the quadratic phase is the one for which the $L^2$ norm is preserved by the spatial change of variables.  A direct computation gives the precise conjugation formula.

Put
\[
 a(t)=\cos t,
 \qquad b(t)=\tan t,
 \qquad \phi(x,t)=\frac{|x|^2}{4}b(t),
\]
so that $b'(t)=a(t)^{-2}$ and
\[
 \partial_t\left(\frac{x}{a(t)}\right)=b(t)\frac{x}{a(t)}.
\]
Moreover,
\[
 \nabla_x=a(t)^{-1}\nabla_y,
 \qquad
 \Delta_x=a(t)^{-2}\Delta_y,
 \qquad
 \nabla_x\phi=\frac{b(t)}2x,
 \qquad
 \Delta_x\phi=\frac d2b(t).
\]
Differentiating \eqref{lensOurNorm} and collecting the transport, amplitude and phase terms yields the exact operator identity
\begin{equation}\label{lensOperatorIdentity}
 \left(i\p_t+\Delta_x-\frac{|x|^2}{4}\right)v(x,t)
 =a(t)^{-d/2-2}e^{-i\phi(x,t)}
 \left(i\p_s+\Delta_y\right)\psi(y,s),
\end{equation}
where $s=\tan t$ and $y=x/a(t)$.  In particular, when $F=0$, the Niederer transform sends the free Schr\"odinger equation to
\begin{equation}\label{HOLens}
 i\p_t v+\Delta_xv-\frac{|x|^2}{4}v=0,
\end{equation}
which is the harmonic-oscillator equation in the normalization \eqref{H}.

The mixed norms are transformed just as explicitly.  Let $(q,r)$ be Schr\"odinger-admissible,
\begin{equation}\label{admissible}
 \frac2q+\frac dr=\frac d2,
 \qquad 2\le q,r\le\infty,\quad (q,r,d)\ne(2,\infty,2).
\end{equation}
From \eqref{lensOurNorm} and the change of variables $y=x/a(t)$,
\begin{equation}\label{spatialLensNorm}
 \|v(t)\|_{L^r_x}
 =a(t)^{d(1/r-1/2)}\|\psi(\tan t)\|_{L^r_y}.
\end{equation}
Since $ds=a(t)^{-2}\,dt$, we obtain
\begin{align}
 \|v\|_{L^q_t(( -\pi/2,\pi/2);L^r_x)}^q
 &=\int_{\mathbb R}a(t(s))^{qd(1/r-1/2)+2}
   \|\psi(s)\|_{L^r_y}^q\,ds \notag\\
 &=\|\psi\|_{L^q_s(\mathbb R;L^r_y)}^q,
 \label{lensnorm}
\end{align}
because \eqref{admissible} is equivalent to
$qd(1/r-1/2)+2=0$ when $q<\infty$; the case $(q,r)=(\infty,2)$ follows directly from \eqref{spatialLensNorm}.  Thus the admissible mixed norm is preserved exactly, not merely up to an inequality.

The same calculation applies to the inhomogeneous term.  If
\[
 i\p_s\psi+\Delta_y\psi=F,
\]
then \eqref{lensOperatorIdentity} gives
\[
 G(x,t)=a(t)^{-d/2-2}e^{-i\phi(x,t)}F(y,s)
\]
in the harmonic-oscillator equation.  Let $(\widetilde q,\widetilde r)$ be another admissible pair and write $\widetilde r'$ for its dual exponent.  Since
\[
 \|G(t)\|_{L^{\widetilde r'}_x}
 =a(t)^{d(1/\widetilde r'-1/2)-2}
   \|F(s)\|_{L^{\widetilde r'}_y},
\]
using $dt=a(t)^2ds$ and
\[
 d\left(\frac1{\widetilde r'}-\frac12\right)=\frac2{\widetilde q},
\]
we find
\begin{equation}\label{forcingLensNorm}
 \|G\|_{L^{\widetilde q'}_tL^{\widetilde r'}_x}
 =\|F\|_{L^{\widetilde q'}_sL^{\widetilde r'}_y}.
\end{equation}
Consequently the standard homogeneous and retarded inhomogeneous free Strichartz estimates are transported exactly to the harmonic oscillator on the interval $(-\pi/2,\pi/2)$.  In particular, the mixed-norm estimates of Ginibre--Velo \cite{GV}, together with the endpoint theorem of Keel--Tao \cite{KT}, give
\begin{equation}\label{HOstrichartz}
 \|v\|_{L^q_tL^r_x}
 \le C\left(\|v(0)\|_{L^2_x}
 +\|G\|_{L^{\widetilde q'}_tL^{\widetilde r'}_x}\right).
\end{equation}

We now return to the Ornstein--Uhlenbeck operator.  Set
\begin{equation}\label{omegaOU}
 \omega(x)=e^{-|x|^2/4},
 \qquad
 d\gamma(x)=(2\pi)^{-d/2}\omega(x)^2\,dx.
\end{equation}
By Proposition \ref{P:connect}, the transformation
\begin{equation}\label{OUtoHOlens}
 v(x,t)=e^{-idt/2}\omega(x)\left(e^{it\mathscr L}\vf\right)(x)
\end{equation}
identifies the Ornstein--Uhlenbeck Schr\"odinger equation \eqref{cp0} with the harmonic-oscillator equation \eqref{HOLens}.  At the level of operators this is the identity
\begin{equation}\label{OUHOoperator}
 \omega\mathscr L\omega^{-1}
 =\Delta-\frac{|x|^2}{4}+\frac d2.
\end{equation}
In particular, multiplication by $(2\pi)^{-d/4}\omega$ is the unitary identification of $L^2(d\gamma)$ with $L^2(dx)$.

The weighted spatial norms that arise in the Gaussian formulation are therefore forced by this conjugation.  For $1\le r<\infty$,
\begin{align}
 \|f\|_{L^r_\gamma(\omega^{r-2})}^r
 &: =\int_{\mathbb R^d}|f(x)|^r\omega(x)^{r-2}\,d\gamma(x)\notag\\
 &=(2\pi)^{-d/2}\int_{\mathbb R^d}|\omega(x)f(x)|^r\,dx,
 \end{align}
and hence
\begin{equation}\label{weightedidentity}
 \|f\|_{L^r_\gamma(\omega^{r-2})}
 =(2\pi)^{-d/(2r)}\|\omega f\|_{L^r(dx)}.
\end{equation}
For $r=\infty$ we use the natural endpoint convention
\begin{equation}\label{weightedidentityInfinity}
 \|f\|_{L^\infty_{\gamma,\omega}}:=\|\omega f\|_{L^\infty(dx)}.
\end{equation}
Thus the one-dimensional admissible endpoint $(q,r)=(4,\infty)$ is covered by the same conjugation argument.  For finite $r$, the same identity with $r$ replaced by $r'$ gives the corresponding input space.  More precisely, with respect to the Gaussian pairing one has
\begin{equation}\label{weighteddual}
\bigl(L^r_\gamma(\omega^{r-2})\bigr)^*\simeq L^{r'}_\gamma(\omega^{r'-2}).
\end{equation}
To verify \eqref{weighteddual}, one should keep track of the powers appearing in H\"older's inequality rather than multiply the two weights directly.  Indeed,
\[
 \frac{r-2}{r}+\frac{r'-2}{r'}=0,
\]
and consequently
\[
 \omega^{(r-2)/r}\,\omega^{(r'-2)/r'}=1.
\]
Thus, for $f\in L^r_\gamma(\omega^{r-2})$ and $g\in L^{r'}_\gamma(\omega^{r'-2})$,
\[
 \int_{\mathbb R^d}|fg|\,d\gamma
 \leq
 \|f\|_{L^r_\gamma(\omega^{r-2})}
 \|g\|_{L^{r'}_\gamma(\omega^{r'-2})}.
\]
This is the same weighted-interpolation mechanism that appears, in a different geometric setting, in Proposition 3.5 of the authors' paper with Staffilani; there Stein's interpolation theorem produces simultaneously the output and input weights before they simplify in the relevant parameter range.  We emphasize this point because the dual weighted spaces should be computed from the weighted pairing, not guessed formally.

Combining \eqref{lensnorm}, \eqref{OUtoHOlens}, and \eqref{weightedidentity} with the classical free Strichartz estimates yields the following theorem.

\begin{theorem}\label{T:OUstrichartztransport}
Let $(q,r)$ and $(\widetilde q,\widetilde r)$ be Schr\"odinger-admissible pairs satisfying \eqref{admissible}.  When $r=\infty$, the notation $L^r_\gamma(\omega^{r-2})$ in \eqref{OUstrichartztransport} is understood as $L^\infty_{\gamma,\omega}$ according to \eqref{weightedidentityInfinity}.  Then, for the Ornstein--Uhlenbeck evolution \eqref{cp0} on $(-\pi/2,\pi/2)$,
\begin{equation}\label{OUstrichartztransport}
 \|e^{it\mathscr L}\vf\|_{L^q_tL^r_\gamma(\omega^{r-2})}
 \le C_{d,q,r}\|\vf\|_{L^2_\gamma},
\end{equation}
and the corresponding inhomogeneous estimate is
\begin{equation}\label{OUstrichartzInhomTransport}
 \left\|\int_0^t e^{i(t-\tau)\mathscr L}F(\tau)\,d\tau\right\|_{L^q_tL^r_\gamma(\omega^{r-2})}
 \le C\|F\|_{L^{\widetilde q'}_tL^{\widetilde r'}_\gamma(\omega^{\widetilde r'-2})},
\end{equation}
with all time integrals restricted to $(-\pi/2,\pi/2)$ and with the usual retarded interpretation of the Duhamel term.
\end{theorem}

It is useful to compare the preceding calculation with the weighted interpolation argument used in Proposition 3.5 of \cite{GSBessel}.  There one interpolates between an $L^1$ weighted estimate and the $L^2$ unitarity estimate.  Stein's theorem produces an output weight and an input weight simultaneously; see, in particular, equations (3.9)--(3.17) there.  In the present Gaussian setting the same bookkeeping is encoded by the conjugation $f\mapsto\omega f$: the output space is $L^r_\gamma(\omega^{r-2})$, while its Gaussian dual is $L^{r'}_\gamma(\omega^{r'-2})$.  Thus the appearance of two apparently different weighted measures in a Strichartz/Duhamel formulation is not an additional dispersive phenomenon, but a consequence of weighted duality.

The normalization used in \cite{DDLS} is obtained from ours by the elementary dilation $y=\sqrt2\,x$.  Their operator
\[
 \mathbf L=-\frac12\Delta+x\cdot\nabla
\]
becomes $-\mathscr L_y$, their Gaussian measure $\pi^{-d/2}e^{-|x|^2}dx$ becomes the Gaussian measure in \eqref{omegaOU}, and their weight $e^{-|x|^2/2}$ becomes $\omega(y)$.  Thus their weighted linear Strichartz estimate is exactly the estimate transported above, after this harmless change of normalization.  More precisely, the diagram
\[
\text{free Schr\"odinger}
\mathrel{\mathop{\longleftrightarrow}\limits^{\text{Niederer}}}
\text{harmonic oscillator}
\mathrel{\mathop{\longleftrightarrow}\limits^{\omega\text{-conjugation}}}
\text{Ornstein--Uhlenbeck}
\]
identifies not only the solution operators but also the admissible mixed norms and their dual forcing norms.  In this precise sense the weighted isotropic Ornstein--Uhlenbeck Strichartz inequalities are the classical free inequalities in a different representation.  The Gaussian factor in $L^r_\gamma(\omega^{r-2})$ is forced by the identity $d\gamma=(2\pi)^{-d/2}\omega^2dx$ and by multiplication by $\omega$ in the OU--harmonic-oscillator conjugation; it does not signal a different dispersive scaling or a new admissibility relation.  This is the Gaussian analogue of the exponential weights arising from the conformal change of variables in \cite[Section~2]{BGstrich}.

The proof in \cite{DDLS} proceeds directly from the Mehler kernel and a $TT^*$ argument.  The point of Theorem \ref{T:OUstrichartztransport} is different: it identifies the underlying inequalities themselves, independently of the proof used to establish them.  We emphasize also that this proposition is not used as an input to the dispersive and uncertainty-principle results of the present paper; its purpose is to delineate precisely which part of the isotropic linear theory is inherited from the classical free Schr\"odinger theory.

\section{Dynamic restriction and the hidden paraboloid}\label{S:restriction}

We next record the restriction-theoretic counterpart of the preceding Strichartz correspondence.  The point is that, in the isotropic model, the frequency-space manifold is not guessed independently: it is inherited from the paraboloid for the classical free Schr\"odinger equation through the exact Niederer and Gaussian conjugations.

We use the term \emph{dynamic restriction} in a spacetime sense.  This
should be distinguished from the recent dynamical restriction principle
of Nicola \cite{NicolaRestriction}, where time remains a parameter and
restriction is imposed, at each fixed time, on curved spatial
submanifolds.  Here the input is a spacetime function $F(x,t)$: after
the transformation $F\mapsto\mathcal C F$, the Fourier transform is
taken in all $d+1$ spacetime variables and restricted to the
Schr\"odinger paraboloid $\Sigma_0\subset\R^{d+1}$.  On the
Ornstein--Uhlenbeck side this produces the time-dependent Gaussian
spacetime measure \eqref{dynamicmeasure}.

Throughout this section assume $d\ge2$ and set
\begin{equation}
 p_{\mathrm{res}}=\frac{2(d+1)}{d+3},
 \qquad
 p_{\mathrm{res}}'=\frac{2(d+1)}{d-1},
 \label{restrictionexponents}
\end{equation}
and let
\begin{equation}
 d\nu(x,t)
 :=(\cos t)^{\frac{2}{d-1}}
 \omega(x)^{\frac{4}{d-1}}\,d\gamma(x)\,dt,
 \qquad |t|<\frac{\pi}{2}.
 \label{dynamicmeasure}
\end{equation}
The power of $\cos t$ is positive: it is exactly the Jacobian weight which makes the $L^{p_{\rm res}'}$ norm invariant under the inverse lens transformation.  Equivalently,
\begin{equation}
 d\nu(x,t)
 =(2\pi)^{-d/2}(\cos t)^{\frac{2}{d-1}}
 \exp\!\left(-\frac{d+1}{2(d-1)}|x|^2\right)dx\,dt.
 \label{dynamicmeasureLeb}
\end{equation}

We use the Fourier transform in spacetime with the normalization already fixed in \eqref{ft}:
\[
\widehat G(\eta,\tau)
 =\int_{\mathbb R^{d+1}}
 e^{-2\pi i(\langle y,\eta\rangle+s\tau)}G(y,s)\,dy\,ds.
\]
The free extension operator associated with the paraboloid is
\begin{equation}
 (E g)(y,s)
 :=\int_{\mathbb R^d}
 e^{2\pi i(\langle y,\xi\rangle-2\pi s|\xi|^2)}g(\xi)\,d\xi.
 \label{freeextension}
\end{equation}
Indeed,
\begin{equation}
 E^*G(\xi)
 =\widehat G\bigl(\xi,-2\pi|\xi|^2\bigr),
 \label{freeRestriction}
\end{equation}
so that the restriction manifold is the paraboloid
\[
 \Sigma_0=\{(\xi,-2\pi|\xi|^2):\xi\in\mathbb R^d\}.
\]

Let $a(t)=\cos t$ and $s=\tan t$, $y=x/a(t)$.  Starting with the free extension $Eg$ and applying first the Niederer transformation and then the Gaussian conjugation gives the Ornstein--Uhlenbeck extension operator
\begin{equation}
 (E_{\rm OU}g)(x,t)
 :=e^{idt/2}\omega(x)^{-1}a(t)^{-d/2}
 e^{-i|x|^2\tan t/4}
 (Eg)\!\left(\frac{x}{a(t)},\tan t\right).
 \label{OUextension}
\end{equation}
The critical exponent $p_{\rm res}'$ satisfies
\[
 d+2-\frac{d p_{\rm res}'}2=-\frac{2}{d-1},
\]
and therefore the exact change of variables gives
\begin{equation}
 \|E_{\rm OU}g\|_{L^{p_{\rm res}'}(d\nu)}
 =C_d\|Eg\|_{L^{p_{\rm res}'}(\mathbb R^{d+1})},
 \label{OUextensionnorm}
\end{equation}
with a fixed dimensional normalization constant stemming only from $d\gamma$.

The adjoint of \eqref{OUextension} can be written explicitly.  We use the sesquilinear pairing $\langle f,g\rangle=\int f\,\overline{g}$ and define, for a spacetime function $F$ on $\mathbb R^d\times(-\pi/2,\pi/2)$,
\begin{equation}
\begin{aligned}
 (\mathcal C F)(y,s)
 :={}&(1+s^2)^{-\frac d4-1-\frac1{d-1}}
 e^{-\frac{id}{2}\arctan s}
 e^{\frac{i s|y|^2}{4(1+s^2)}}
 \\
 &\times
 e^{-\frac{d+3}{4(d-1)}\frac{|y|^2}{1+s^2}}
 F\!\left(\frac{y}{\sqrt{1+s^2}},\arctan s\right).
\end{aligned}
\label{OUrestrictiontransform}
\end{equation}
Then a direct change of variables in the adjoint pairing yields
\begin{equation}
 (E_{\rm OU}^*F)(\xi)
 =(2\pi)^{-d/2}
 \widehat{\mathcal C F}\bigl(\xi,-2\pi|\xi|^2\bigr).
 \label{OUrestrictionoperator}
\end{equation}

The following theorem gives the precise form of Theorem A.

\begin{theorem}[Dynamic restriction for the isotropic Ornstein--Uhlenbeck flow]
\label{T:OUrestriction}
For every $F\in C_0^\infty(\mathbb R^d\times(-\pi/2,\pi/2))$,
\begin{equation}
 \left\|
 (2\pi)^{-d/2}
 \widehat{\mathcal C F}\bigl(\xi,-2\pi|\xi|^2\bigr)
 \right\|_{L^2(\mathbb R^d)}
 \le C_d
 \|F\|_{L^{p_{\rm res}}(d\nu)}.
 \label{OUrestriction}
\end{equation}
Equivalently, the restriction operator for the isotropic Ornstein--Uhlenbeck Schr\"odinger flow is the adjoint of the transported extension operator \eqref{OUextension}, and the restriction manifold is precisely the classical paraboloid written in the free variables.
\end{theorem}

\begin{proof}
By the classical restriction theorem for the paraboloid, see \cite{Tomas,Stri}, or equivalently its dual extension estimate,
\begin{equation}
 \|Eg\|_{L^{p_{\rm res}'}(\mathbb R^{d+1})}
 \le C_d\|g\|_{L^2(\mathbb R^d)}.
 \label{TomasSteinextension}
\end{equation}
Equation \eqref{OUextensionnorm} transfers \eqref{TomasSteinextension} to the Ornstein--Uhlenbeck extension operator.  Taking the Banach adjoint and using \eqref{OUrestrictionoperator} gives \eqref{OUrestriction}.  Thus the inequality is exactly the classical Tomas--Stein theorem expressed in the Ornstein--Uhlenbeck variables.
\end{proof}

\begin{remark}\label{R:OUrestriction}
The spacetime measure \eqref{dynamicmeasure} is not the invariant Gaussian measure $d\gamma\,dt$.  Its time factor and Gaussian spatial factor are forced by the lens Jacobian and by the Gaussian conjugation.  The theorem is therefore the classical Tomas--Stein theorem expressed through the full free--harmonic--Ornstein--Uhlenbeck correspondence.  What is specific to the present formulation is the resulting spacetime restriction operator in Gaussian variables: its adjoint is the ordinary $(d+1)$-dimensional Fourier transform of $\mathcal C F$ restricted to the continuous Schr\"odinger paraboloid.
\end{remark}

\section{Critical nonlinear dynamics}\label{S:nonlinear}

We next turn to the nonlinear consequences of the full correspondence. The critical variational structure identifies the canonical weighted nonlinearity, and the well-posedness theory is then transported from the classical mass-critical equation.

\subsection{The critical variational structure}\label{S:EL}
The preceding equivalence extends naturally to the variational problem associated with the sharp Strichartz constant.  We record this observation since it further clarifies that the free, harmonic-oscillator and isotropic Ornstein--Uhlenbeck formulations are different representations of the same dispersive structure.

Let $I=(-\pi/2,\pi/2)$ and let $(q,r)$ be an admissible pair.  For $f\neq0$ set
\begin{equation}\label{OUsharpfunctional}
 \mathcal S_{q,r}(\vf)
 :=\frac{\|e^{it\mathscr L}\vf\|_{L^q_t(I;L^r_\gamma(\omega^{r-2}))}}
 {\|\vf\|_{L^2_\gamma}}.
\end{equation}
Assume that the optimal constant in \eqref{OUstrichartztransport} is attained at a sufficiently regular nonzero function $\vf$, and write
\[
 u(t)=e^{it\mathscr L}\vf.
\]
By homogeneity we may impose $\|\vf\|_{L^2_\gamma}=1$ and vary
\[
 \int_I\left(\int_{\mathbb R^d}|u(x,t)|^r\omega(x)^{r-2}\,d\gamma(x)\right)^{q/r}dt.
\]
A standard first-variation argument gives the Euler--Lagrange equation
\begin{equation}\label{OUEL}
 \int_I e^{-it\mathscr L}\left[
 \|u(t)\|_{L^r_\gamma(\omega^{r-2})}^{q-r}
 \omega^{r-2}|u(t)|^{r-2}u(t)
 \right]dt
 =\lambda \vf,
\end{equation}
for a real constant $\lambda>0$.  Formula \eqref{weighteddual} is precisely the weighted duality underlying this variation.

Under the unitary Gaussian conjugation
\[
 v(x,t)=e^{-idt/2}\omega(x)u(x,t),\qquad g=(2\pi)^{-d/4}\omega \vf,
\]
the equation \eqref{OUEL} becomes, up to the harmless normalization constants already displayed in \eqref{weightedidentity}, the Euler--Lagrange equation for the sharp harmonic-oscillator Strichartz functional.  Applying the Niederer transformation then gives the classical free Schr\"odinger extremizer equation.  Thus the variational problems, and not only the linear estimates, are transported by the diagram above.

The diagonal admissible exponent
\begin{equation}\label{criticaldiag}
 q=r=\frac{2(d+2)}{d}
\end{equation}
is particularly transparent.  Since $r-2=4/d$, \eqref{OUEL} reduces to
\begin{equation}\label{OUELcritical}
 \int_I e^{-it\mathscr L}
 \left[\omega^{4/d}|u(t)|^{4/d}u(t)\right]dt
 =\lambda \vf.
\end{equation}
The same Gaussian power is obtained by transporting the standard mass-critical nonlinear Schr\"odinger equation
\begin{equation}\label{freecriticalNLS}
 i\partial_s\psi+\Delta\psi=\mu|\psi|^{4/d}\psi
\end{equation}
through the Niederer transformation and then through \eqref{OUtoHOlens}.  Indeed, the mass-critical exponent is exactly the one for which the time-dependent factor generated by the lens transform cancels, and one obtains
\begin{equation}\label{OUcriticalNLS}
 i\partial_t \vf+\mathscr L \vf
 =\mu\omega^{4/d}|\vf|^{4/d}\vf.
\end{equation}
Consequently, the same critical Gaussian nonlinearity is singled out independently by the transformation of the classical mass-critical NLS and by the Euler--Lagrange equation of the sharp diagonal OU Strichartz functional.

The preceding observation leads to a nonlinear Cauchy problem which is intrinsic to the correspondence developed above.  We emphasize that its Gaussian power is not chosen independently: it is forced by the transport of \eqref{freecriticalNLS}.

\subsection{Well-posedness and global dynamics}\label{S:OUNLS}
Let
\[
 \rho=\frac{2(d+2)}{d}.
\]
For an interval $I\subset\mathbb R$ set
\[
 X(I)=C(I;L^2_\gamma)\cap L^\rho\bigl(I;L^\rho_\gamma(\omega^{\rho-2})\bigr).
\]
Notice that $\rho-2=4/d$.  The weighted duality in \eqref{weighteddual} and the identity \eqref{weightedidentity} give the critical nonlinear estimate
\begin{equation}\label{OUNLestimate}
 \|\omega^{4/d}|u|^{4/d}u\|_{L^{\rho'}_tL^{\rho'}_\gamma(\omega^{\rho'-2})}
 \leq C_d\|u\|_{L^\rho_tL^\rho_\gamma(\omega^{\rho-2})}^{1+4/d}.
\end{equation}
Indeed, after multiplication by $\omega$ this is exactly the ordinary Lebesgue estimate for $|\omega u|^{4/d}\omega u$.  The same observation gives the corresponding Lipschitz estimate for the difference of two nonlinearities.

The local theory is the standard critical Strichartz fixed-point theory transported to the present variables; see, for example, \cite{GV,GVstrich,Caze} for the classical nonlinear Schr\"odinger framework.

Weighted nonlinearities can of course arise for other structural reasons.  For example, in ongoing work with Stewart \cite{GSprep}, an equivariant reduction of the standard planar NLS followed by a conjugation to a Bessel operator produces a power nonlinearity with an explicit radial weight.  In that setting, as in \eqref{OUcriticalNLS}, the weight is forced by the transformation rather than postulated independently.  We mention this only to emphasize the general principle that weighted nonlinearities often encode the geometry of the representation in which the equation is written.

The local assertion is contained in the following result.

\begin{theorem}\label{T:OUNLSlwp}
Let $\mu\in\mathbb R$ and $u_0\in L^2(\Rm,d\gamma)$.  There exists an open interval $I\ni0$ and a unique mild solution
\[
 u\in X(I)
\]
of
\begin{equation}\label{OUNLScp}
 \begin{cases}
 i\partial_tu+\mathscr Lu=\mu\omega^{4/d}|u|^{4/d}u,\\
 u(0)=u_0.
 \end{cases}
\end{equation}
The solution depends continuously on the initial datum.  Moreover, on every Niederer interval
\[
 I_k=\left(k\pi-\frac{\pi}{2},k\pi+\frac{\pi}{2}\right)
\]
the equation \eqref{OUNLScp} is equivalent, by the Gaussian conjugation and the Niederer transformation, to the standard mass-critical NLS \eqref{freecriticalNLS} on the full free time axis.
\end{theorem}

\begin{proof}
The local assertion on $I_0=(-\pi/2,\pi/2)$ follows either directly
from Theorem \ref{T:OUstrichartztransport}, \eqref{OUNLestimate},
and the standard contraction argument, or by transporting the
classical $L^2$ local theory for \eqref{freecriticalNLS}; see, for
instance, \cite{Caze}. The same calculation, after translating time
by $k\pi$, gives the local equivalence on each Niederer interval
$I_k$ once the data at one endpoint of the interval are prescribed.
The critical exponent is precisely the one for which the time-dependent
factor generated by the lens transformation cancels.

\end{proof}

 The endpoints $t=k\pi+\pi/2$ of the Niederer intervals are not
caustics of the linear Mehler kernel; they correspond instead to
$s=\pm\infty$ in the free variables. The interpretation of scattering
in terms of convergence at the endpoint of a lens chart is classical;
see, in particular, \cite{Carles,TaoLens}, and also \cite{BGTV} for
an application of the lens transform to scattering in $\Sigma$ for
mass-subcritical nonlinear Schr\"odinger equations. The corresponding
chart transition is governed by scattering for the free mass-critical
NLS. More precisely, if a global free solution $\psi$ has scattering
states $\psi_\pm\in L^2(\Rm)$, then the lens transform has strong
$L^2$ limits as $t\to\pm\pi/2$, and these limits are obtained from
$\psi_\pm$ by the Gaussian-conjugated quarter-period transform.
Thus a boundary datum at one end of a Niederer interval determines the
incoming scattering state for the adjacent interval. In the
defocusing mass-critical problem the wave and scattering operators
are global on $L^2(\Rm)$, so this transition can be iterated across all
chart boundaries.

The global assertion is the following.

\begin{theorem}[Global defocusing well-posedness]\label{T:OUdefocusing}
Assume that $\mu>0$.  Then, for every $u_0\in L^2(d\gamma)$, the Cauchy problem \eqref{OUNLScp} admits a unique global mild solution
\[
 u\in C(\mathbb R;L^2_\gamma)\cap L^\rho_{\rm loc}\bigl(\mathbb R;L^\rho_\gamma(\omega^{4/d})\bigr),
 \qquad \rho=\frac{2(d+2)}d.
\]
For every $k\in\mathbb Z$, its restriction to the Niederer interval
\[
 I_k=\left(k\pi-\frac\pi2,k\pi+\frac\pi2\right)
\]
is obtained, through the Gaussian conjugation and the Niederer transformation, from a global scattering solution of the defocusing mass-critical NLS on $\mathbb R_s$.  In particular, the solution extends uniquely in $L^2(d\gamma)$ through every boundary $t=k\pi+\pi/2$ of the Niederer charts, and the global solution depends continuously on the initial datum.
\end{theorem}

\begin{proof}
The defocusing mass-critical NLS is globally well-posed and scattering
in $L^2(\Rm)$; see Dodson \cite{Do1,Do2,Do3}. Transport the unique
free solution by the Niederer map and the Gaussian conjugation on
$I_0=(-\pi/2,\pi/2)$. By Theorem \ref{T:OUNLSlwp}, this gives the
unique solution of \eqref{OUNLScp} on $I_0$. As $s\to\pm\infty$,
scattering gives free asymptotic states in $L^2$, and hence strong
$L^2(d\gamma)$ limits of the corresponding OU solution at the two
endpoints $t=\pm\pi/2$. These endpoint values are related to the
free scattering states by the Gaussian-conjugated quarter-period
transform described in Corollary \ref{C:quarter}. Applying the global
free wave operator to the outgoing state gives the unique free
solution on the next chart with the required incoming state.
Iterating this construction over $k\in\mathbb Z$ produces a global
OU solution whose restrictions to the intervals $I_k$ agree at their
common endpoints. Uniqueness and continuous dependence follow from
the corresponding properties on each free chart and from the
continuity of the scattering and wave operators in $L^2$.
\end{proof}

\begin{remark}\label{R:focusing}
For the focusing sign the correspondence remains exact, but a global statement requires the corresponding free mass-critical solution to be global and scattering.  In particular, the familiar scattering/blow-up alternatives for the focusing free equation are transported to \eqref{OUNLScp}.  This nonlinear issue should not be confused with the genuine linear caustics $t=k\pi$, where the nondegenerate Mehler kernel representation breaks down although the linear unitary group itself remains well defined.
\end{remark}

\section{Uncertainty principles}\label{S:uncertainty}

We exploit Proposition \ref{P:semigroup} to prove Proposition \ref{P:main}. We mention that related dynamical forms of Hardy's uncertainty principle for
Schr\"odinger equations with real drift were recently obtained by
Garofalo and Lunardi \cite{GarLun}, in the more general setting of
possibly degenerate operators satisfying the H\"ormander
hypoellipticity condition. Their result is formulated in Lebesgue space and involves the
covariance geometry generated by the drift and diffusion matrices,
whereas the results below concern the unitary isotropic
Ornstein--Uhlenbeck evolution in Gaussian space and exploit its
exact correspondence with the harmonic oscillator.

We will need the following result, see \cite[Theorem 1.1]{CEKPV}. We note for the reader that our normalisation of the Fourier transform
\begin{equation}\label{ft}
\hat \vf(\xi) = \mathscr F(\vf)(\xi) = \int_{\Rm} e^{-2\pi i\sa \xi,x\da} \vf(x) dx,
\end{equation}
differs from theirs, and this accounts for the different constants in \eqref{hardyL2} below.

\begin{theorem}\label{T:har2}
Assume that $h:\Rm\to \mathbb C$ is a measurable function that satisfies 
\begin{equation}\label{hardyL2}
||e^{a|\cdot|^2} h||_{L^2(\Rm)} + ||e^{b|\cdot|^2} \hat h||_{L^2(\Rm)}<\infty.
\end{equation}
If $a b\ge \pi^2$, then $h\equiv 0$.  
\end{theorem}

We are ready to give the 

\begin{proof}[Proof of Proposition \ref{P:main}]
Let $\vf\in L^2(\Rm,d\gamma)$. By the assumption
$ab\sin^2s\geq 1/16$, one has $\sin s\neq0$. Write
\[
s=k\pi+\tau,
\qquad k\in\mathbb Z,
\qquad 0<\tau<\pi.
\]
By the group property and the parity identity \eqref{parity-pre},
\[
(e^{is\mathscr L}\vf)(x)
=(e^{i\tau\mathscr L}\vf)((-1)^k x).
\]
Since the Gaussian weights in \eqref{L2} are invariant under parity,
it therefore suffices to prove the result with $s$ replaced by
$\tau\in(0,\pi)$. With $\psi$ as in \eqref{psi}, consider
\begin{equation}\label{ht}
h_\tau(x)=e^{i\frac{\cot\tau |x|^2}{4}}\psi(x).
\end{equation}
We have
\begin{align}\label{one}
\int_{\Rm}e^{2a|x|^2}|h_\tau(x)|^2dx
&=\int_{\Rm}e^{2a|x|^2}|\psi(x)|^2dx
=\int_{\Rm}e^{2a|x|^2}|\vf(x)|^2e^{-\frac{|x|^2}{2}}dx
\\
&=(2\pi)^{\frac d2}\|e^{a|x|^2}\vf\|^2_{L^2(\Rm,d\gamma)}<\infty,
\notag
\end{align}
by \eqref{L2}. Since $\tau\in(0,\pi)$, Proposition \ref{P:semigroup}
gives
\begin{equation*}
\left|\widehat h_\tau\left(\frac{x}{4\pi\sin\tau}\right)\right|
=(4\pi)^{\frac d2}(\sin\tau)^{\frac d2}
e^{-\frac{|x|^2}{4}}|(e^{i\tau\mathscr L}\vf)(x)|.
\end{equation*}
Hence
\begin{align}\label{beauty}
&\left(\int_{\Rm}e^{2b|x|^2}
\left|\widehat h_\tau\left(\frac{x}{4\pi\sin\tau}\right)\right|^2dx\right)^{1/2}
\\
&=(2\pi)^{\frac d4}(4\pi)^{\frac d2}(\sin\tau)^{\frac d2}
\|e^{b|x|^2}e^{i\tau\mathscr L}\vf\|_{L^2(\Rm,d\gamma)}<\infty.
\notag
\end{align}
After the change of variables $x=4\pi\sin\tau\,y$,
\begin{equation}\label{nik}
\int_{\Rm}e^{2(16b\pi^2\sin^2\tau)|y|^2}
|\widehat h_\tau(y)|^2dy<\infty.
\end{equation}
Applying Theorem \ref{T:har2} to $h_\tau$ and using
$ab\sin^2\tau=ab\sin^2s\geq1/16$, we conclude that
$h_\tau\equiv0$. Hence $\psi\equiv0$ by \eqref{ht}, and therefore
$\vf\equiv0$.
\end{proof}



\vskip 0.2in

\section{Declarations}

\noindent \textbf{Data availability statement:} This manuscript has no associated data.

\vskip 0.2in

\noindent \textbf{Funding and/or Conflicts of interests/Competing interests statement:} The author declares that he does not have any conflict of interest for this work

\bibliographystyle{amsplain}

\end{document}